\documentclass[11pt]{amsart}

\usepackage{amssymb}
\usepackage[colorlinks=true, urlcolor=rltblue, citecolor=drkgreen, linkcolor=drkred] {hyperref}
\usepackage{color}
\usepackage{pdfsync}
\definecolor{rltblue}{rgb}{0,0,0.4}
\definecolor{drkgreen}{rgb}{0,0.4,0}
\definecolor{drkred}{rgb}{0.5,0,0}

\newtheorem{thm}{Theorem}[section]
\newtheorem{lemma}[thm]{Lemma}

\newtheorem{theorem}[thm]{Theorem}

\theoremstyle{definition}
\newtheorem{definition}[thm]{Definition}

\theoremstyle{remark}

\newtheorem{remark}[thm]{Remark}

\newtheorem{historic}[thm]{Historic Remark}

\theoremstyle{plain}

\newcounter{contenumi}

\def\leqt{\leq_T}
\def\geqt{\geq_T}

\def\teq{\equiv_T}

\def\upto{\mathop{\upharpoonright}}
\def\la{\langle}
\def\ra{\rangle}
\def\and{\mathrel{\&}}

\def\sminus{\smallsetminus}

\def\Si{\Sigma}

\newcommand\rightdate[1]{\footnotetext{  Saved: #1 \\ Compiled: \today}}

\def\om{\omega}

\def\si{\sigma}
\def\b{\beta}

\def\a{\alpha}
\def\b{\beta}

\def\QQ{{\mathbb Q}}
\def\M{{\mathcal M}}

\def\rk{rk}

\def\g{\gamma}
\def\d{\delta}

\def\Si{\Sigma}

\def\S{{\mathcal S}}

\def\Sp {\mathop{\mathrm{Sp}}}
\def\om{\omega}

\def\Sp{Sp}

\def\Equ{\ E\ }
\newcommand{\Equu}[1]{\ E_{#1}\ }

\def\siha{\hat{\si}}
\def\tauha{\hat{\tau}}
\def\That{\hat{T}}

\title{Analytic equivalence relations satisfying hyperarithmetic-is-recursive}

\author{Antonio Montalb\'an}
\thanks{The author was partially supported by NSF grant DMS-0901169 and the Packard Fellowship.
The author would also like to thank Richard A. Shore for useful conversations, and Andrew Marks for commenting on an earlier draft of this paper.
This paper was written while the author participated in the Buenos Aires 
Semester in Computability, Complexity and Randomness, 2013.
}

\address{Department of Mathematics\\
University of California, Berkely\\
 USA}

\email{antonio@math.berkeley.edu}
\urladdr{\href{http://www.math.berkeley.edu/~antonio/index.html}{www.math.berkeley.edu/$\sim$antonio}}

\begin{document}

\rightdate{June 5, 13 -- submitted}
\maketitle


%

\begin{abstract}
We prove, in ZF+$\bf\Sigma^1_2$-determinacy, that for any analytic equivalence relation $E$, the following three statements are equivalent:
(1) $E$ does not have perfectly many classes, (2) $E$ satisfies hyperarithmetic-is-recursive on a cone, and (3) relative to some oracle, for every equivalence class $[Y]_E$ we have that a real $X$ computes a member of the equivalence class if and only if $\om_1^X\geq\om_1^{[Y]}$.

We also show that the implication from (1) to (2) is equivalent to the existence of sharps over $ZF$.
\end{abstract}

\section{Introduction}

In 1955 Clifford Spector \cite{Spe55} proved that every well-ordering of $\om$ with a hyperarithmetic presentation has a computable presentation.
This theorem has been of great importance in recursion theory, in lightface descriptive set theory, etc.
In this paper we prove that Spector's theorem can be extended to very general circumstances which apply to a variety of known cases, unearthing a more general phenomenon that is behind all of them.

Some years ago, the author \cite{MonEqui, MonBSL} showed that Spector's theorem can be extended to the class of all linear ordering if we replace isomorphism by bi-embeddability:
Every hyperarithmetic linear ordering is bi-embeddable with a computable one.
Notice that among well-orderings, the notions of isomorphism and bi-embeddability coincide, so Spector's theorem is a special case of this more general result.
Not much later, Greenberg and the author showed the same result for bi-embeddability of $p$-groups \cite{GMRanked}.
Let us remark that for both, countable linear orderings and countable $p$-groups, the number of equivalence classes under bi-embeddability is $\aleph_1$, as proved by Laver \cite{Lav71}, and Barwise and Eklof \cite{BE70} respectively.
Some time later, the author showed that any counter-example to Vaught's conjecture, that is a theory which has $\aleph_1$ but not continuum many models, if exists, it would also satisfy the same property \cite{MonVC}, giving a computability theoretic statement equivalent to Vaught's conjecture.
After all these examples we started to think that something more general was going on.

\begin{definition}
We say that an equivalence relation $E$ on the reals, $2^\om$, satisfies {\em hyperarithmetic-is-recursive} if every hyperarithmetic real is $E$-equivalent to a computable one.
\end{definition}

Our main result says that any analytic equivalence relation with less than continuum many equivalence classes {\em essentially} satisfies hyperarithmetic-is-recursive.
We say ``essentially'' because one can alway build a non-natural equivalence relation for which this is not true.
To overcome this problem we ask for the equivalence relation to satisfy hyperarithmetic-is-recursive relative to almost every oracle, where ``almost every'' is in the sense of Martin's measure.
If we have a natural equivalence relation at hand, one would expect to be able to prove either that it satisfies hyperarithmetic-is-recursive or that it does not, and in either case, one would expect that this proof to relativize to every oracle.
Therefore, restricting oneself to almost every oracle should not make a difference on natural equivalence relations.

By {\em Martin's measure} we mean the $\{0,1\}$-measure, where a set of reals has {\em Martin's measure 1} if it contains a cone, where a {\em cone} a set of reals of the form $\{X\in 2^\om: X\geqt Y\}$ for some $Y$ called the {\em base of the cone}.
D.\ A.\ Martin showed that this is a measure on the degree-invariant sets of reals of complexity $\Gamma$, assuming $\Gamma$-determinacy, where $\Gamma$ is a complexity class like for instance Borel, analytic, etc.

\begin{definition}
We say that an equivalence relation on $2^\om$ satisfies {\em hyperarithmetic-is-recursive on a cone} if there is a $C\in 2^\om$ (the base of the cone) such that for every $X$ which computes $C$, every $X$-hyperarithmetic real is equivalent to an $X$-computable one.
\end{definition}

Here is our main theorem.

\begin{theorem}\label{thm: main}
(ZF+$\mathbf\Si^1_2$-Det)
Let $E$ be an analytic equivalence relation on $2^\om$.
The following are equivalent:
\begin{enumerate}\renewcommand{\theenumi}{H\arabic{enumi}}
\item There is no perfect set all whose elements are $E$-inequivalent.    \label{pa: no perfect}
\item $E$ satisfies hyperarithmetic-is-recursive on a cone.      \label{pa: hyp is recursive}
\item There is an oracle relative to which, for every $Y\in 2^\om$, the {\em degree spectrum} of its equivalence class, $\Sp([Y]_E)$, is of the form $\{X\in 2^\om: \om_1^X\geq\a\}$ for some ordinal $\a\in\om_1$.   \label{pa: low is recursive}
\end{enumerate}
\end{theorem}

Burgess \cite{Bur78} showed that given an analytic equivalence relation, either it has at most $\aleph_1$ many equivalence classes, or it has perfectly many classes (i.e.\ there is a perfect set of $E$-inequivalent reals).
Thus, if the continuum hypothesis is false, saying that $E$ does not have perfectly many classes is equivalent to saying that it has $\leq\aleph_1$ many classes.
The existence of such a perfect set is absolute--it is $\Si^1_2$--and does not depend on the continuum hypothesis.

The {\em degree spectrum} of an equivalence class is the analog of the degree spectrum of a structure, a notion widely studied in Computable Structure Theory.
It gives us a way of measuring the complexity of the equivalence class in terms of how difficult it is to compute a member.
More precisely, define
\[
\Sp([Y]_E) = \{X\in 2^\om: \exists W \leqt X\ (W\Equ Y)\}.
\]
The set $\{X\in 2^\om: \om_1^X\geq\a\}$ is the set of all reals that can compute copies of all ordinals below $\a$.
It is a very particular set, and the fact that the spectrum of any equivalence class would have this form seems to be a very strong statement.
Let us remark that the relativized version of the spectrum is defined as follows: $\Sp^Z([Y]_E)= \{X\in 2^\om: \exists W\leqt X\oplus Z\ (W\Equ Y)\}$ and that the set $\{X: \om_1^X\geq\a\}$ relativized to $Z$ becomes $\{X: \om_1^{X\oplus Z}\geq\a\}$.

Let us observe that this result applies to all the examples mentioned before.
For instance, let $X\ E_{\om_1}\ Y$ if either neither of $X$ and $Y$ is coding a well-ordering of $\om$, or the orderings they code are isomorphic.
This is a $\Si^1_1$ equivalence relation with one equivalence class for each countable ordinal, and one equivalence class for all the reals not coding a well-ordering.
It has $\aleph_1$ equivalence classes, and by Spector's theorem it satisfies hyperarithmetic-is-recursive.
We can do the same with bi-embeddability of linear orderings or $p$-groups, which we know have $\aleph_1$ equivalence classes.
So, Theorem \ref{thm: main} tells us that they satisfy hyperarithmetic-is-recursive on a cone.
The proofs in \cite{MonEqui,GMRanked} proved these results relative to every oracle, and not just on a cone.
Our general proof does not say anything about what happens relative to every oracle other than we expect the behavior to be the same relative to every oracle and relative to almost every oracle if the relation is natural enough.
The proofs in \cite{MonEqui,GMRanked} still require a deep analysis of the embeddability relation among linear orderings and $p$-groups used for those results.
In \cite{MonVC} the author showed that any counter-example to Vaught's conjecture must satisfy hyperarithmetic-is-recursive on a cone, and that result follows directly from Theorem \ref{thm: main}.
However, the proof in \cite{MonVC} is much more constructive, and analyses the structure among the models of counter-example to Vaught's conjecture, something we do not get from the proof in this paper.

\

Theorem \ref{thm: main} uses ZF+$\mathbf\Si^1_2$-Determinacy.
That (\ref{pa: low is recursive}) implies (\ref{pa: hyp is recursive}), and that (\ref{pa: hyp is recursive}) implies (\ref{pa: no perfect}), can be proved in just $ZF$.
The use of ZF+$\mathbf\Si^1_2$-Determinacy is only necessary to show that (\ref{pa: no perfect}) implies (\ref{pa: low is recursive}).
That (\ref{pa: no perfect}) implies (\ref{pa: hyp is recursive}) only requires $\mathbf\Si^1_1$-Determinacy, which is equivalent to  the existence of sharps ($\forall X\ (X^\sharp $ exists)), as proved by Harrington \cite{Har78}.
We show that the use of $\mathbf\Si^1_1$-Determinacy is actually necessary:

\begin{theorem}(ZF)\label{thm: reversal}
The following statements are equivalent:
\begin{enumerate} \renewcommand{\theenumi}{O\arabic{enumi}}
\item Every lightface $\Si^1_1$ equivalence relation without perfectly many classes satisfies hyperarithmetic-is-recursive on a cone. \label{pa: 1 implies 2}
\item $0^\sharp$ exists. \label{pa: zero sharp}
\end{enumerate}
\end{theorem}

This theorem will be proved in Section \ref{se: reversal}.

\

An interesting remark about our main theorem \ref{thm: main} is that it shows how cardinality issues get reflected at the hyperarithmetic/computable level.

\section{The proof of the main theorem}

We start by proving the following effective version of Burgess' Theorem \cite[Corollary 1]{Bur79}.

\begin{lemma}\label{le: Burguess}
For every $\Si^1_1$ equivalence relation $E$ there is a decreasing nested sequence of equivalence relations $\{E_\a:\a\in \om_1\}$ such that $E_\a$ is $\Si^0_{\a+1}$ uniformly in $\a$, and $E=\bigcap_{\a\in\om_1} E_\a$.
\end{lemma}
\begin{proof}
Using Kleene's normal form, let $T$ be a computable sub-tree of $2^{<\om}\times \om^{<\om}\times 2^{<\om}$ such that for all $X,Y$, if we let 
\[
T_{X,Y}=\{\si\in\om^{<\om}: (X\upto|\si|,\si,Y\upto|\si|)\in T\},
\]
then $X\Equ Y$ if and only if $T_{X,Y}$ is ill-founded.
The first wrong idea would be to let $E_\a=\{(X,Y): \rk(T_{X,Y})\geq\a\}$, which is known to be $\Si^0_{\a+1}$ uniformly in $\a$ and satisfies $E=\bigcap_{\a\in\om_1} E_\a$.
Unfortunately $E_\a$ might not be transitive or symmetric.
In Burgess' proof \cite{Bur79} he shows that, for a club of ordinals $\a$, $E_\a$ is an equivalence relation, which is all he needs to get his result.
This is not enough for our more effective version.

To get the symmetry property, let us replace $T$ by the tree $T\cup \{(\tau,\si,\rho): (\rho,\si,\tau)\in T\}$.
This way we get that $T_{X,Y}=T_{Y,X}$, and we still have that $X\Equ Y\iff \neg WF(T_{X,Y})$.

We will modify the tree even further to get transitivity.
For each $k\geq 1$ and $X,Y\in 2^\om$, let 
\begin{multline*}
T^k_{X,Y} =\{ (\si_1,\tau_1,\si_2,\tau_2,...,\tau_{k-1},\si_k)\in \om^n\times 2^{n}\times\cdots\times 2^n\times \om^n: \\
		 n\in\om, (X\upto n,\si_1,\tau_1)\in T, (\tau_1,\si_2,\tau_2)\in T,...,(\tau_{k-1},\si_k,Y\upto n)\in T\}.
\end{multline*}
Note that $T^1_{X,Y}=T_{X,Y}$.
Let $\That_{X,Y} = \sum_{k\in\om} T^k_{X,Y}$, that is the disjoint union of all the $T^k_{X,Y}$ identifying the roots of all these trees.
We note that $X\Equ Y\iff \neg WF(\That_{X,Y})$:
This is because if there is a path through one of the $T^k_{X,Y}$, then we would have $(Z_1,X_1,...,X_{k-1},Z_k)$ such that, for all $i$, $Z_{i+1}\in T_{X_i,X_{i+1}}$ where $X_0=X$ and $X_{k+1}=Y$, and hence  $X=X_0\Equ X_1\Equ X_2 \Equ ...\Equ X_{k+1}=Y$.
On the other hand, if $X\Equ Y$, then $T^1_{X,Y}$ is ill-founded, and hence so is $\That_{X,Y}$.

We are now ready to define $E_\a$ as follows.
Let 
\[ 
X \Equu\a Y\iff \rk(\That_{X,Y})\geq \a.
\]
We still have that $X\Equ Y\iff (\forall\a<\om_1)\ X\Equu\a Y$, that these relations are nested, and that they are uniformly $\Si^0_{\a+1}$.
We now claim that each $E_{\a}$ is an equivalence relation.
They are reflexive just because $E$ is.
It is not hard to see that $\rk(\That_{X,Y})=\rk(\That_{Y,X})$, and hence that $E_\a$ is  symmetric.

To prove transitivity suppose that $X \Equu\a Y \Equu\a Z$.
Then, since $\rk(\That_{X,Y})=\sup\{\rk(T^k_{X,Y}):k\in\om\}$, for every $\b<\a$ there exist $k,l\in\om$, $\rk(T^k_{X,Y})\geq\b$ and $T^l_{Y,Z}\geq\b$.
We claim that $T^{k+l}_{X,Z}\geq\b$, which would imply that $\rk(\That_{X,Y})\geq\a$ and hence that $X\Equu\a Z$ as needed.
For each $(\si_1,\tau_1,\si_2,\tau_2,...,\si_k)\in T^k_{X,Y}$ and $(\siha_1,\tauha_1,\siha_2,\tauha_2,...,\siha_l)\in T^l_{Y,Z}$ of the same length $n$, we note that $(\si_1,\tau_1,\si_2,\tau_2,...,\si_k,Y\upto n, \siha_1,\tauha_1,\siha_2,\tauha_2,...,\siha_l)\in T^{k+l}_{X,Z}$.
This is an order-preserving embedding from $\{(\rho,\pi)\in T^k_{X,Y}\times T^l_{Y,Z}: |\rho|=|\pi|\}$ into $T^{k+l}_{X,Z}$.
It follows that $\rk(T^{k+l}_{X,Z})\geq \min\{T^k_{X,Y}, T^l_{Y,Z}\}$, and hence that $T^{k+l}_{X,Z}\geq\b$ as wanted.
\end{proof}

\begin{remark}\label{rmk: bound}
Notice that if $X\Equu{\om_1^{X\oplus Y}} Y$, then $X\Equ Y$.
This is because $\That_{X,Y}$ is computable in ${X\oplus Y}$, and hence, if it is well-founded, it has rank below $\om_1^{X\oplus Y}$.
\end{remark}

The following is the key lemma to prove the main direction of Theorem \ref{thm: main}.
We will then apply Turing determinacy to the set considered in the lemma, or to a variation of it, to get what we want.
Recall that, for a complexity class $\Gamma$, {\em $\Gamma$-Turing determinacy} says that any degree-invariant $\Gamma$-set of reals $\S$ which is co-final in the Turing degrees contains a cone.
(A set $\S$ is {\em degree invariant} if $\forall X\teq Y\ (X\in \S \leftrightarrow Y\in S)$, and it is {\em co-final} if $\forall Z\in 2^\om\ \exists X\geqt Z\ (X\in \S)$.)
$\Gamma$-Turing determinacy is due to D.\ A.\ Martin, and follows from plain $\Gamma$-determinacy. 

\begin{lemma}(ZF) \label{le: cofinal}
For every analytic equivalence relation $E$ without perfectly many classes, 
the set $\S\subseteq 2^\om$, defined as follows
\[
\S=\{X\in 2^\om: \forall Y\ (\om_1^{X\oplus Y}=\om_1^X \Rightarrow X\in \Sp([Y]_E)\},
\]
is co-final in the Turing degrees.
\end{lemma}
\begin{proof}
To prove that $\S$ is co-final, take any $Z$, and let us build $X\in \S$ with $X\geqt Z$.
By relativizing the rest of the argument, let us assume that $Z$ is computable and that $E$ is lightface $\Si^1_1$, and hence that the tree $\That$ used in Lemma \ref{le: Burguess} is computable.

For each $\a$, there is no perfect set of $E_\a$-inequivalent reals, as otherwise there would be one for $E$.
Silver \cite{Sil80} showed that any Borel equivalence relation without perfectly many classes has countably many classes.
Thus, each $E_\a$ has countably many classes.
For each $\a\in\om_1$, let $\la A_{\a,n}:n\in\om\ra\subseteq 2^\om$ be a list which contains one real of each $E_\a$-equivalence class.
(For the reader who worries about the use of choice, we will see how to avoid it later.)
Let us code this whole sequence as a single subset $A$ of $\om_1\times\om\times\om$:
Just let $(\a,n,m)\in A$ if and only if $m\in A_{\a,n}$.
Recall G\"odel's hierarchy $L_\a[A]$, where $A$ is considered as a relation symbol and $L_{\a+1}[A]$ consist of the first-order definable subsets of $(L_\a[A]; \in,A\cap \a\times\om\times\om)$ (see, for instance, \cite[Section 1.3]{Kan03}).
For some $\a\in\om_1$ we have that $L_{\a}[A]$ is admissible, and that every $\b<\a$ can be coded by a well-ordering of $\om$ within $L_{\a}[A]$.
(For instance, take any $\a$ where $L_\a[A]$ is an elementary substructure of $L_{\om_1^{L[A]}}[A]$.)
Now, using Barwise compactness for the admissible set $L_{\a}[A]$ \cite[Theorem III.5.6]{Bar75} we get an ill-founded model $\M=(M;\in^\M,A^\M)$ of $KP$ whose ordinals have well-founded part equal to $\a$, with $A^\M\upto \a$ coinciding with $A\upto\a$, and satisfying that every ordinal can be coded by a real.
(To show this one has to consider the infinitary theory in the language $L=\{\in, A,c\}$ saying all this, plus axioms saying that the constant symbol $c$ is an ordinal and that any ordinal below $\a$ exists and that $c$ is above it.
Then  observe that whole the set of axioms is $\Si_1(L_\a[A])$, and that, choosing $c$ appropriately, $L_\a[A]$ is a model of any subset of these axioms which is a set in $L_\a[A]$.
Thus, by Barwise compactness \cite[Theorem III.5.6]{Bar75} this theory has a model and its ordinals have well-founded part at least $\a$.
Then, using \cite[Theorem III.7.5]{Bar75}, we get such a model with well-founded part exactly $\a$.)
Let $\a^*\in ON^\M\sminus \a$, and let $X$ be a real in $\M$ coding $\a^*$ and $A^\M\upto \a^*$.
Notice that $\om_1^X=\a$.
(To see this, we have that $\om_1^X\geq\a$ because it codes every initial segment of $\a$, and $\om_1^X\leq\a$ because every $X$-computable well-ordering is isomorphic to an ordinal in $\M$.)

We claim that $X\in \S$.
Consider $Y$ with $\om_1^{X\oplus Y}\leq\a$; We must show that $X$ computes a real $E$-equivalent to $Y$.
Let us think of $\a^*$ as the well-ordering of $\om$ of type $\a^*$ which is coded by $X$.
Let
\[
P=\{\b\in \a^*: (\exists W\leqt X)\ W\Equu\b Y\}.
\]
(Let us remark that when $\b$ is not an true ordinal, i.e.\ $\b\in\a^*\sminus\a$, we can still talk about $\Equu\b$ using the definition from Lemma \ref{le: Burguess}, that is, $X\Equu\b Y \iff \rk(\That_{X,Y})\geq\b \iff \exists f\colon \b\to \That_{X,Y}\ (\forall \g,\d<\b\ (f(\g)\subsetneq f(\d) \rightarrow \g>\d))\}$.)
The set $P\subseteq\om$ is $\Si^1_1(X,Y)$, using this $\Si^1_1$ definition of $E_\b$.
The set $P$ contains all the true ordinals  $\b<\a$ because $X$ computes all the reals $A_{\b,n}$, which are taken one from each $E_\b$-equivalence class. 
We can now apply an overspill argument:
Since $\om_1^{X\oplus Y}\leq\a$, $\a$ (viewed as the initial segment of the presentation of $\a^*$) is not $\Si^1_1(X\oplus Y)$ (as, being the well-ordered part of $\a^*$ is  $\Pi^1_1(X)$, and it cannot be $\Delta^1_1(X,Y)$).
Thus, there must exist a non-standard $\b^*\in P\sminus \a$.
Let $Y^*$ be the witness that $\b^*\in P$.
That is $Y^*\leqt X$ and  $Y^*\Equu{\b^*}Y$.
By the nestedness of these equivalence relations, for all  true ordinals $\b<\a$, $Y^*\ E_{\b}\ Y$.
Since $\om_1^{Y\oplus Y^*}\leq\om_1^{Y\oplus X}=\a$, by Remark \ref{rmk: bound}, we have that $Y^*\Equ Y$ as needed to get that $X\in \S$.

For the interested reader, let us see how to avoid the use of the axiom of choice.
This proof uses the axiom of choice only to define the sequence $A_{\b,n}$, which can be defined directly as follows.
By Shoenfield's absoluteness, for each $\b<\om_1^L$, the sequence $\la A_{\b,n}:n\in\om\ra$ can be taken to be inside $L_{\om_1^L}$, and hence we can define it as the $<_L$-least such that $\forall Y\exists n\ (Y\Equu\b A_{\b,n}) \and \forall n,m\ \neg(A_{\b,n}\Equu\b A_{\b,m})$.
This definition works inside $L_{\om_1^L}$, and hence $(L_{\om_1^L};\in,A)$ is admissible, and we can let $\a=\om_1^L$.
(Unless the reader is worried that for this lemma we might have $\om_1^L=\om_1$, in which case  
  any ordinal $\a$ with $L_\a[A]$ an elementary substructure of  $(L_{\om_1^L};\in,A)$ would work.)
\end{proof}

We are now ready to prove the main theorem.
Let us start by showing that if $E$ does not have perfectly many classes, then $E$ satisfies hyperarithmetic-is-recursive on a cone.

\begin{proof}[Proof of (\ref{pa: no perfect}) $\Rightarrow$ (\ref{pa: hyp is recursive}) in (ZF+$\mathbf\Si^1_1$-Det)]
Consider the set $\S_1$ of the oracles relative to which $E$ satisfies hyperarithmetic-is-recursive,
that is
\[
\S_1=\{X\in 2^\om: \forall Y\leq_{hyp} X \ \exists W\leqt X\ (W \Equ Y))\},
\]
where $Y\leq_{hyp} X$ means that $Y$ is hyperarithmetic in $X$.
This set is $\Si^1_1$ as the quantifier $\forall Y\leq_{hyp}X$ can be replaced by an existential quantifier over all the reals (see \cite[Exersice III.3.11]{Sac90}).
The set $\S_1$ is clearly degree invariant. 
Also, it contains the set $\S$  because, by Spector's theorem, $Y\leq_{hyp}X\Rightarrow \om_1^{X\oplus Y}=\om_1^X$, and hence by Lemma \ref{le: cofinal}, it is co-final in the Turing degrees.
By $\Si^1_1$-Turing Determinacy, which follows from $\Si^1_1$-Determinacy, it contains a cone.
\end{proof}

Let us now show that if $E$ does not have perfectly many classes all the degree spectra of the $E$-equivalence classes are of the form $\{X:\om_1^X\geq \a\}$.

\begin{proof}[Proof of (\ref{pa: no perfect}) $\Rightarrow$ (\ref{pa: low is recursive}) in (ZF+$\mathbf\Si^1_2$-Det)]
Consider the set $\S$ from Lemma \ref{le: cofinal}.
This set is $\Pi^1_2$ and degree invariant.
(We are using that the relation $\om_1^Y=\om_1^X$ is $\Si^1_1$, as it says that every $Y$-computable well-ordering is isomorphic to an $X$-computable ordering, and vice-versa.
Is easy to see that ``$X\in \Sp([Y]_E)$'' is $\Si^1_1$.)
So, by $\Si^1_2$-Turing-deteminacy, which follows from $\Si^1_2$-Detetmiancy, we have that $\S$ contains a cone.

Relativize the rest of the proof to the base of this cone, and hence assume that every real belongs to $\S$.
Take  $Y\in 2^\om$. 
We claim that
\[
\Sp([Y]_E)=\{X\in 2^\om:\om_1^X\geq \om_1^{[Y]}\},
\]
where $\om_1^{[Y]}=\min\{\om_1^W:W\Equ Y\}$.
It is clear from the definition of $\om_1^{[Y]}$ that if $\om_1^{[Y]}>\om_1^X$, then $X$ computes no real $E$-equivalent to $Y$.
Suppose now that $\om_1^{[Y]}\leq\om_1^X$--we need to show that $X$ computes a real $E$-equivalent to $Y$.
Assume, without loss of generality, that $Y$ is so that $\om_1^Y=\om_1^{[Y]}$ (otherwise, replace it by an $E$-equivalent real with this property).

If $\om_1^{X\oplus Y}=\om_1^{X}$,  we would be done because $X\in\S$.
Otherwise, let $Z\geqt Y$ be such that $\om_1^Z=\om_1^X$.
It is proved in \cite[Theorem 2.10]{Har78} (and with a different proof in \cite[Lemma 3.6]{MonVC}), that if  $\om_1^Z=\om_1^X$ there exists a real $G$ with 
\[
\om_1^{X}=\om_1^{X\oplus G} =\om_1^G =\om_1^{G\oplus Z} = \om_1^{Z}.
\]
Since $G\in \S$ and $\om_1^{G\oplus Y}\leq\om_1^{G\oplus Z}=\om_1^G$, there is a $Y_1\leqt G$ such that $Y_1\Equ Y$.
Therefore $\om_1^{Y_1\oplus X}\leq  \om_1^{G\oplus X} =\om_1^{X}$.
Since $X\in \S$, $X$ computes $Y_2$ such that $Y_2 \Equ Y_1$ and hence $Y_2\Equ Y$.
\end{proof}

It not hard to see that (\ref{pa: low is recursive}) implies (\ref{pa: hyp is recursive}).

\begin{proof}[Proof of (\ref{pa: hyp is recursive}) $\Rightarrow$ (\ref{pa: no perfect}) in ZF]
Suppose there is a perfect tree $R\subseteq 2^{<\om}$ all whose paths are $E$-inequivalent.
We need to show that relative to every oracle on a cone, there is a hyperarithmetic real not $E$-equivalent to any computable real.
By relativizing the rest of the proof, assume that this oracle and $R$ are both computable.

First, let us observe that for some $\a<\om_1^{CK}$, all the paths through $R$ are not only $E$-inequivalent, but also $E_\a$-inequivalent:
For each $X,Y\in [R]\times [R]$ with $X\neq Y$ there is an ordinal $\b$ such that $\neg( X\Equu\b Y)$, namely the rank of $\That_{X,Y}$ plus 1 (where $\That_{X,Y}$ is as in Lemma \ref{le: Burguess}).
Thus, $\That$ gives us a computable map from $[R]\times [R]\sminus \{(X,X):X\in [R]\}$ to the class of well-founded trees.
By $\Si^1_1$-boundedness (due to Spector \cite{Spe55}), the ranks of these trees are all bounded below some ordinal  $\a\in\om_1^{CK}$.

Let $G$ be an $(\a+1)$-Cohen-generic real (i.e.\ it decides every $\Si^0_{\a+1}$ formula) computable from $0^{(\a+2)}$, and let $R(G)$ be the path through $R$ following $G$ at every split.
So $R(G)$ is hyperarithmetic.
We claim that is is not $E$-equivalent to any computable real.
Suppose it is, that $X$ is computable and $X\Equ R(G)$.
Since all the paths are $E_\a$-inequivalent, for any other path $Z\in [R]$, $Z\neq R(G)$ we have that $\neg(Z \Equu\a X)$.
The real $G$ can then be defined as the unique real such that $R(G)\Equu\a X$, which is a $\Si^0_{\a+1}$ formula.
By $\a+1$-genericity, there is a condition $p\in 2^{<\om}$ forcing that $G$ satisfies this formula.
But then every other $\a+1$-generic extending $p$ would satisfy this formula too, contradicting the uniqueness of $G$.
\end{proof}

\section{A reversal}   \label{se: reversal}

In this section we show that the use of $\Si^1_1$-determinacy  in proving that (\ref{pa: no perfect}) implies (\ref{pa: hyp is recursive}) is not only sufficient but also necessary.
We do not know, however, if the use of $\Si^1_2$-determinacy  in proving that (\ref{pa: no perfect}) implies (\ref{pa: low is recursive}) is necessary.

Let us remark that when $E$ is a lightface-$\Si^1_1$ equivalence relation, our proof of (\ref{pa: no perfect}) $\Rightarrow$ (\ref{pa: hyp is recursive}) only uses lightface $\Si^1_1$-determinacy, which is equivalent to the existence of $0^\sharp$.
Thus, have we already proved that (\ref{pa: zero sharp}) implies (\ref{pa: 1 implies 2}) in Theorem \ref{thm: reversal}.

Before proving the theorem, let us review a key lemma by Sami \cite{Sam99}.
First, define 
\[
\S=\{Y\in 2^\om: \exists Z\in 2^\om\ (\om_1^Z=\om_1^Y \and \forall W\leq_{hyp} Z\ (W\leqt Y))\}.
\]
Sami showed that if $\S$ contains a cone, then $0^\sharp$ exists:
He showed \cite[Proposition 3.8]{Sam99} that if $\S$ contains the cone with base $C$, then every $C$-admissible ordinal is a cardinal in $L$, which then implies that $0^\sharp$ exists by a result of Silver \cite[Section 1]{Har78}.

\begin{proof}[Proof of (\ref{pa: 1 implies 2})$\Rightarrow$(\ref{pa: zero sharp})]
To prove that $0^\sharp$ exists, we will prove that the set $\S$ above contains a cone.
For this, we will define a $\Si^1_1$ equivalence relation $E$ without perfectly many classes, and then show that the cone relative to which $E$ satisfies hyperarithmetic-is-recursive is contained in $\S$. 

Let $R$ be the set of all reals coding a structure isomorphic to  $(L_\a(A);\in)$ for some ordinal $\a\in\om_1$ and some $A\subseteq\om$.
This set is $\Pi^1_1$ since to verify that a model is a presentation of $L_\a(A)$ all one needs to do is check well-foundedness, and then check that each level is defined from the previous one correctly.

Consider the equivalence relation $E$ that holds of presentations of the structures $L_{\a_X}(A_X)$ and $L_{\a_Y}(A_Y)$ respectively if $\a_X=\a_Y$ and $\om_1^{A_X}=\om_1^{A_Y}$, and which lets all the reals outside $R$ be equivalent to each other.
This relation is $\Si^1_1$, since $R$ is $\Pi^1_1$, deciding if $\a_X=\a_Y$ is $\Si^1_1$ and deciding if $\om_1^{A_X}=\om_1^{A_Y}$ is also $\Si^1_1$.
This equivalence relation has $\aleph_1$ equivalence classes, one for each value of the pair $(\a_X,\om_1^{A_X})$.
Since this is true in any model of ZF, $E$ cannot contain perfectly many classes classes (because having perfectly many classes is a $\Si^1_2$ statement).
So, by (\ref{pa: 1 implies 2}), $E$ must satisfy hyperarithmetic-is-recursive on a cone, say with base $C$.
Take $Y\geqt C$--we need to show that $T\in \S$.
For each $\a<\om_1^Y$ there is a presentation of $(L_\a(Y);\in)$ which is hyperarithmetic-in-$Y$.
But then $Y$ computes a real $E$-equivalent to this presentation, that is, a presentation of $(L_{\a}(Z),\in)$ for some $Z$ with $\om_1^Z=\om_1^Y$.
Let $\a^*$ be a presentation of the Harrison linear ordering \cite{Har68} relative to $Y$, that is, a $Y$-computable linear ordering isomorphic $\om_1^Y+\om_1^Y\cdot\QQ$.
Let 
\[
P=\{\b\in \a^*: Y \mbox{ computes a presentation of $(L_{\b}(Z);\in)$ for some $Z$ with } \om_1^Z=\om_1^Y\}.
\]
This set is $\Si^1_1(Y)$ as, given $\b$, checking that a structure is a presentation of $(L_{\b}(Z);\in)$ is hyperarithmetic, and checking if $\om_1^Z=\om_1^Y$ is $\Si^1_1$.
By our comments before, the set $P$ contains all $\b$ in the well-founded part of $\a^*$, namely $\om_1^Y$.
Therefore, by an overspill argument, $P$ must contain some non-standard $\b^*\in\a^*\sminus \om_1^Y$.
Let $Z^*$ be such that $Y$ computes a copy of $(L_{\b^*}(Z^*);\in)$.
Every real $W$ which is hyperarithmetic in $Z^*$ belongs to $L_{\b}(Z^*)$ for some $\b<\om_1^Y$ and hence belongs this presentation of $(L_{\b^*}(Z^*);\in)$ too.
Therefore, $W\leqt Y$.
We have shown that  $\forall W\leq_{hyp} Z^*\ (W\leqt Y))$ as needed to get that $Y\in \S$.
\end{proof}

\bibliography{bftypes}

\begin{thebibliography}{Mon13}

\bibitem[Bar75]{Bar75}
Jon Barwise.
\newblock {\em Admissible sets and structures}.
\newblock Springer-Verlag, Berlin, 1975.
\newblock An approach to definability theory, Perspectives in Mathematical
  Logic.

\bibitem[BE71]{BE70}
Jon Barwise and Paul Eklof.
\newblock Infinitary properties of abelian torsion groups.
\newblock {\em Ann. Math. Logic}, 2(1):25--68, 1970/1971.

\bibitem[Bur78]{Bur78}
John~P. Burgess.
\newblock Equivalences generated by families of {B}orel sets.
\newblock {\em Proc. Amer. Math. Soc.}, 69(2):323--326, 1978.

\bibitem[Bur79]{Bur79}
John~P. Burgess.
\newblock A reflection phenomenon in descriptive set theory.
\newblock {\em Fund. Math.}, 104(2):127--139, 1979.

\bibitem[GM08]{GMRanked}
Noam Greenberg and Antonio Montalb\'an.
\newblock Ranked structures and arithmetic transfinite recursion.
\newblock {\em Transactions of the AMS}, 360:1265--1307, 2008.

\bibitem[Har68]{Har68}
J.~Harrison.
\newblock Recursive pseudo-well-orderings.
\newblock {\em Transactions of the American Mathematical Society},
  131:526--543, 1968.

\bibitem[Har78]{Har78}
Leo Harrington.
\newblock Analytic determinacy and {$0\sp{\sharp }$}.
\newblock {\em J. Symbolic Logic}, 43(4):685--693, 1978.

\bibitem[Kan03]{Kan03}
Akihiro Kanamori.
\newblock {\em The higher infinite}.
\newblock Springer Monographs in Mathematics. Springer-Verlag, Berlin, second
  edition, 2003.
\newblock Large cardinals in set theory from their beginnings.

\bibitem[Lav71]{Lav71}
Richard Laver.
\newblock On {F}ra\"\i ss\'e's order type conjecture.
\newblock {\em Annals of Mathematics (2)}, 93:89--111, 1971.

\bibitem[Mon05]{MonEqui}
Antonio Montalb\'an.
\newblock Up to equimorphism, hyperarithmetic is recursive.
\newblock {\em Journal of Symbolic Logic}, 70(2):360--378, 2005.

\bibitem[Mon07]{MonBSL}
Antonio Montalb\'an.
\newblock On the equimorphism types of linear orderings.
\newblock {\em Bulletin of Symbolic Logic}, 13(1):71--99, 2007.

\bibitem[Mon13]{MonVC}
Antonio Montalb{\'a}n.
\newblock A computability theoretic equivalent to {V}aught's conjecture.
\newblock {\em Adv. Math.}, 235:56--73, 2013.

\bibitem[Sac90]{Sac90}
Gerald~E. Sacks.
\newblock {\em Higher recursion theory}.
\newblock Perspectives in Mathematical Logic. Springer-Verlag, Berlin, 1990.

\bibitem[Sam99]{Sam99}
Ramez~L. Sami.
\newblock Analytic determinacy and {$0^{\sharp}$}. {A} forcing-free proof of
  {H}arrington's theorem.
\newblock {\em Fund. Math.}, 160(2):153--159, 1999.

\bibitem[Sil80]{Sil80}
Jack~H. Silver.
\newblock Counting the number of equivalence classes of {B}orel and coanalytic
  equivalence relations.
\newblock {\em Ann. Math. Logic}, 18(1):1--28, 1980.

\bibitem[Spe55]{Spe55}
Clifford Spector.
\newblock Recursive well-orderings.
\newblock {\em Journal Symbolic Logic}, 20:151--163, 1955.

\end{thebibliography}
\bibliographystyle{alpha}

\end{document}